\documentclass{birkjour}

\usepackage{amsmath, amssymb, amsthm}
\usepackage{mathrsfs, mathtools}
\usepackage{hyperref}
\usepackage{cleveref}

\theoremstyle{plain}
\newtheorem{theorem}{Theorem}[section]

\theoremstyle{definition}
\newtheorem{definition}[theorem]{Definition}

\theoremstyle{remark}
\newtheorem{remark}[theorem]{Remark}
\newcommand{\HH}{\mathbf{H}}          
\newcommand{\HC}{\mathcal{H}_C}
\newcommand{\HA}{\mathcal{H}_A}
\newcommand{\KK}{\mathcal{K}}         
\newcommand{\KKtot}{\mathcal{K}_{\mathrm{tot}}} 
\newcommand{\fh}{\mathfrak{h}}        
\newcommand{\Jbf}{\mathbf{J}}         
\newcommand{\ran}{\operatorname{ran}}
\newcommand{\dom}{\operatorname{dom}}
\newcommand{\spect}{\sigma} 

\newcommand{\ip}[2]{\langle #1,\, #2 \rangle}
\newcommand{\ipC}[2]{\langle #1,\, #2 \rangle_C}
\newcommand{\ipA}[2]{\langle #1,\, #2 \rangle_A}
\newcommand{\ipK}[2]{\langle #1,\, #2 \rangle_{\KK}}

\begin{document}

\title[Simultaneous Williamson's normal form for positive operators]{Simultaneous Williamson's normal form for positive operators}
\author[Kamat]{Rudra Kamat}
    \address{Department of Mathematics, Heriot-Watt University, Edinburgh, Scotland EH14 4AS}
    \email{rk4043@hw.ac.uk}
\author[Tonny]{Anmary Tonny}
	\address{Statitics and Mathematics Unit, Indian Statistical Institute Bangalore Centre, 8th Mile Mysore Road, Bengaluru, India 560059}
	\email{anmarytonny97@gmail.com}
\date{May 5, 2026}

\begin{abstract}
We prove a simultaneous Williamson's normal form for two symplectically commuting positive operators $A, B$ on a real symplectic Hilbert space $(\HH, J)$. Under a natural commutativity condition, there exists a (generally unbounded) operator $\Phi$ that simultaneously diagonalizes $A$ and $B$ in the symplectic setting as $A = \Phi^t M_A \Phi$, $B = \Phi^t M_B \Phi$, and maps $J$ to a canonical symplectic operator. The result generalizes Williamson's normal form to infinite dimensions and yields a complete spectral invariant for the pair $(A,B)$.
\end{abstract}

\maketitle
\tableofcontents

\section{Introduction}
\label{sec:intro}

Williamson's normal form is a fundamental result in symplectic linear algebra. It is a diagonalization result for positive-definite real matrices of even order, where the diagonalization is carried out by symplectic matrices. Note that a symplectic matrix is a real matrix $M$ of order $2n,$ that satisfies the identity $M^{t}JM = J$, where
 	$ J = \begin{bmatrix}
 		0 & I_n \\
 		-I_n & 0
 	\end{bmatrix}$ and $I_n
 	$ is the identity matrix of order $n$. The normal form is as follows.
    
\begin{theorem} (Williamson's normal form \cite{jw01})  \label{wnffd}
 		Let $A$ be a strictly positive matrix of order $2n$. Then there exists a symplectic matrix $M$ of order $2n$ and a strictly positive diagonal matrix D of order $n$ such that
 		$$
 		A = 
 		M^t\begin{bmatrix}
 			D & 0 \\
 			0 & D
 		\end{bmatrix}M.
 		$$
 		Furthermore, D is unique up to the order of its entries.		
 	\end{theorem}  The entries of the diagonal matrix $D$ are defined as the symplectic eigenvalues of the matrix $A$. Williamson's normal form and the theory of symplectic eigenvalues have found applications in matrix analysis, symplectic geometry, and mathematical physics; see \cite{arvindpramana, rbt01, bhatia_jain_2021,  mag02, mag01, repsumtiju, logmaj, wolf, jain2022derivatives, mishra2024generalization, hkm20, Par10, Par13, Par15, tyt01}. The finite-dimensional result was extended to positive invertible operators on real separable Hilbert space by Bhat and Tiju in 2019, and is as follows.
\begin{theorem}(Williamson's normal form \cite{BJ19}) \label{wnfifd}
 		Let $\mathcal{H}$ be a real separable Hilbert space and $A$ be a strictly positive invertible operator on $ \mathcal{H} \oplus \mathcal{H}$. Then there exists a positive invertible operator $D$ on $\mathcal{H}$ and a symplectic operator $M$ on $\mathcal{H} \oplus \mathcal{H}$ such that 
 		$$
 		A = M^t\begin{bmatrix}
 			D & 0 \\
 			0 & D
 		\end{bmatrix}M.$$ 
 		Furthermore, $D$ is unique up to conjugation with an orthogonal transformation. 
 	\end{theorem}

     \noindent The spectrum of $D$ is defined as the symplectic spectrum of $A$. Also, a bounded linear operator $Q$ on  $\mathcal{H} \oplus \mathcal{H}$ is called a symplectic operator if $Q^{t}JQ = J$, where $J = \begin{bmatrix}
         0 & I \\
         -I & 0
     \end{bmatrix}$ and $I$ is the identity operator on $\mathcal{H}$.
 
While the finite-dimensional theory is well-established and studied, the infinite-dimensional counterpart is in its initial stage, a few of the contributions being \cite{bhatia2021szegHo, tcjvbkatp, vbkatp}. The purpose of this paper is to develop an infinite-dimensional counterpart of the simultaneous Williamson's normal form which was established by Kamat and Mishra in \cite{KAMAT2026298}. More generally, they obtained conditions for simultaneous symplectic reduction of positive-semidefinite quadratic forms to normal forms in a common symplectic basis. In particular, their results identify symplectic commutativity as the appropriate analogue of the usual commutativity condition for simultaneous orthogonal diagonalization.

The transition from matrices to operators is, however, not merely a change of setting. In finite dimensions, Williamson's normal form is based on the existence of eigenvalues and corresponding eigenvectors and ultimately yields a finite symplectic basis in which the operator has a diagonal matrix representation. In infinite dimensions, a positive operator need not have a discrete spectrum or a complete system of eigenvectors. Thus, a direct extension of the finite-dimensional eigenvector construction is generally not available. The appropriate analogue of a diagonal matrix is instead a multiplication operator arising from a spectral representation. Moreover, the constructions involved in an infinite-dimensional symplectic normal form naturally lead to densely defined and potentially unbounded operators, making questions concerning domains and closures. 

Here, we consider two positive operators $A$ and $B$ on a real separable Hilbert space satisfying the appropriate symplectic commutation condition. Rather than seeking a common symplectic eigenbasis, which need not exist in infinite dimensions, we formulate simultaneous symplectic diagonalization in terms of a common spectral representation by multiplication operators. Our result recovers the finite-dimensional simultaneous Williamson's normal form at the level of spectral representations, while addressing features that are specific to the infinite-dimensional setting. In finite dimensions, simultaneous diagonalization is expressed in terms of a common symplectic basis and a finite collection of symplectic
eigenvalues. In contrast, the infinite-dimensional result must accommodate continuous spectral components and the possible absence of eigenvectors. Consequently, diagonal matrices are replaced by multiplication operators and symplectic eigenvalues are replaced by measurable spectral parameters. In addition, the construction requires the treatment of densely defined and unbounded operators and their domains. These issues are absent from the finite-dimensional formulation and constitute a central part of the operator-theoretic construction developed in this paper.

We proceed in two stages. We first establish a symplectic spectral decomposition for a single positive operator. We then use this decomposition, together with the commutation relation between the two operators, to construct a common symplectic spectral representation for the pair $(A,B)$. Throughout the paper, $\mathcal H$ denotes a real separable Hilbert space. We work on the doubled Hilbert space \(
\HH:=\mathcal H\oplus\mathcal H,
\)
equipped with the standard symplectic form \(
\omega(x,y):=\langle x,Jy\rangle,
\)
where \(
J:=
\begin{bmatrix}
0&I\\
-I&0
\end{bmatrix}
\in\mathcal B(\HH).
\) Since
\(
J^t=-J,
J^2=-I,
\)
the bilinear form $\omega$ is skew-symmetric and non-degenerate. Consequently, $(\HH,\omega)$ is a real symplectic Hilbert space.




\noindent 

\section{Symplectic Spectral Decomposition via Direct Integrals}
\label{sec:direct-integral}

In this section, we prove that every bounded positive operator $A$ on a real symplectic Hilbert space, whose kernel is $J$-invariant, admits a direct integral decomposition in which the symplectic spectrum of $A$ appears explicitly and the symplectic structure is preserved. 

\subsection{Discussion on the main result and outline of the proof} \label{secproofoutline}
We outline the proof strategy before presenting the details. Since $\HH_0 := \ker(A)$ is $J$-invariant, $J$ preserves $\HH_R := \overline{\ran A} = (\ker A)^\perp$, giving $J = J_0\oplus J_R$ where $J_0 =
J|_{\HH_0}$ and $J_R = J|_{\HH_R}$. The central object is the Hilbert space $\HA$ obtained by completing $\HH_R$ in the inner product induced by $A_R := A|_{\HH_R}$. The operator $T_R := J_R A_R$ then extends to a bounded skew-adjoint operator $\widehat{T}$ on $\HA$; the spectral theorem applied to $-\widehat{T}^2$ simultaneously yields the direct integral decomposition and identifies the symplectic spectrum as a multiplication operator $M$. The symplectic transformation $\Phi_R := \Gamma_0 \circ \iota$ is then the composition of the inclusion $\iota : \HH_R \hookrightarrow \HA$ with the spectral isomorphism $\Gamma_0 : \HA \xrightarrow{\sim} \KK\oplus\KK$, which is bounded. Its transpose $\Phi_R^t : \dom(\Phi_R^t) \subset \KK\oplus\KK \to \HH_R$ is densely defined and unbounded when $A_R$ does not have a bounded inverse. The normal form becomes
\[
  A_R = \Phi_R^t \Phi_R,
  \qquad
  \Phi_R J_R \Phi_R^t = M\,\Jbf_\KK,
\]
where $M = \int^\oplus_{(0,\infty)} \lambda\cdot
I_{\fh_\lambda\oplus\fh_\lambda}\,d\mu(\lambda)$ is the Williamson's multiplication operator. The factor $M$ appears in the symplectic identity rather than in the normal form because $\Phi_R$ is an isometry from $(\HH_R, \ipA{\cdot}{\cdot})$ to $(\KK\oplus\KK, \ipK{\cdot}{\cdot})$, and the $A$-inner product already encodes the symplectic spectrum. We show in \cref{rem:two-forms} that this is equivalent to the more familiar form $A_R = (\Phi_R^\prime)^t M \Phi_R^\prime$,
$\Phi_R^\prime J_R (\Phi_R^\prime)^t = \Jbf_\KK$, obtained by setting $\Phi_R^\prime = M^{-1/2}\Phi_R$, a densely defined unbounded operator on $\HH_R$. Finally, we adjoin the kernel part to obtain the Williamson's normal form for the operator $A$.

Our main result, Theorem~\ref{thm:sim-wnf}, provides a simultaneous symplectic representation of the form
\[
A=\Phi^t M_A\Phi,
\qquad
B=\Phi^t M_B\Phi,
\qquad
\Phi J\Phi^t=J_{\mathcal K_{\mathrm{tot}}},
\]
where $M_A$ and $M_B$ are multiplication operators on the corresponding spectral representation space. More precisely, the multiplication functions associated with $M_A$ and $M_B$ are denoted by $\alpha$ and
$\beta$, respectively, and the relevant joint spectral parameter is \(
\lambda=\alpha+\beta.
\)
Thus, the role played in finite dimensions by a common diagonal matrix is assumed here by a pair of multiplication operators on a common spectral space.

\subsection{Definitions}

Before going to the Williamson's normal form, we recall the following definitions.

\begin{definition}[Measurable Hilbert space bundle]
\label{def:mhsb}
Let $(X, \mathscr{B}, \mu)$ be a $\sigma$-finite Borel
measure space. A \emph{measurable Hilbert space bundle} over $X$ is an assignment $\lambda \mapsto \fh_\lambda$ of separable real Hilbert spaces together with a countable measurable partition $X = \bigsqcup_{n\geq 1}X_n$ (with $X_n \in \mathscr{B}$) such that $\dim\fh_\lambda = n$, for all $\lambda \in X_n$. The \emph{direct integral} $\KK = \int^\oplus_X \fh_\lambda\,d\mu(\lambda)$ is the Hilbert space of measurable sections $v(\lambda)\in\fh_\lambda$ with
$\int_X\|v(\lambda)\|^2\,d\mu(\lambda)<\infty$ and the inner product is given by 
\begin{equation}\ipK{u}{v} \coloneqq \int_X \ip{u(\lambda)}{v(\lambda)}_{\fh_\lambda}\,d\mu(\lambda).\end{equation}

A bounded operator $T$ on $\KK$ is \emph{decomposable} if $(Tv)(\lambda) = T(\lambda)v(\lambda)$ $\mu$-a.e.\ for bounded operators $T(\lambda)$ on $\fh_\lambda$; denoted as \begin{equation}T = \int^\oplus_{X}
T(\lambda)\,d\mu(\lambda).\end{equation}
\end{definition}

\begin{definition}[Fibrewise symplectic form]
\label{def:fiberwise-symp}
Given a bundle $\lambda\mapsto\fh_\lambda$ over $(X,\mu)$, set $\KK\cong \int^\oplus_X (\fh_\lambda\oplus\fh_\lambda) \,d\mu(\lambda)$. The \emph{fibrewise symplectic form} is $\langle \cdot, \Jbf_\KK \cdot \rangle$, where
\[
  \Jbf_\KK \coloneqq \int^\oplus_X J_\lambda\,d\mu(\lambda),
  \qquad
  J_\lambda \coloneqq
  \begin{pmatrix}0 & I_{\fh_\lambda}\\
    -I_{\fh_\lambda} & 0\end{pmatrix},
\]
satisfying $\Jbf_\KK^t = -\Jbf_\KK$ and
$\Jbf_\KK^2 = -\mathrm{I}_{\KK}$.
\end{definition}
\subsection{Williamson's normal form for bounded positive operators - direct integral approach}

In this section, we prove the Williamson's normal form for a bounded positive operator $A$ ($\langle Ax, x \rangle \geq 0,$ for all $x \in \HH$) on $\HH$ whose kernel is $J-$invariant. Following the notations as in the beginning of this section, we state the result below.

\begin{theorem}[Williamson's Normal Form] \label{thm:di-wnf}
Let $(\HH, \langle \cdot, \cdot \rangle, J)$ be a real separable symplectic Hilbert space, and let $A \in \mathcal{B}(\HH)$ be a bounded positive operator such that $\ker(A)$ is $J$-invariant. Then there exist:
\begin{enumerate}
    \item A $\sigma$-finite Borel measure $\mu$ on $[0, \infty)$ and a measurable Hilbert space bundle $\lambda \mapsto \mathfrak{h}_\lambda$ over $[0, \infty)$ with direct integral space 
    \[
    \mathcal{K} = \int_{[0,\infty)}^\oplus \mathfrak{h}_\lambda \, d\mu(\lambda);
    \]
    \item A real Hilbert space $\mathfrak{h}_0 \subset \HH_0$ such that $\HH_0 = \mathfrak{h}_0 \oplus J_0 \mathfrak{h}_0$;
    \item A closed, densely defined linear operator 
    \[
    \Phi = \Phi_0 \oplus \Phi_R' : \operatorname{dom}(\Phi) \subset \HH \longrightarrow \mathcal{K}_{\mathrm{tot}},
    \]
    where $\mathcal{K}_{\mathrm{tot}} = (\mathfrak{h}_0 \oplus \mathfrak{h}_0) \oplus (\mathcal{K} \oplus \mathcal{K})$, $\Phi_0 : \HH_0 \to \mathfrak{h}_0 \oplus \mathfrak{h}_0$ is a bounded orthogonal isomorphism, and $\Phi_R' = M^{-1/2}\Phi_R$ is a densely defined, closable operator on $\mathcal{H}_R$ with domain $\operatorname{dom}(\Phi_R') = \{ x \in \mathcal{H}_R : \Phi_R x \in \mathrm{ran}(M^{1/2}) \}$;
\end{enumerate}
such that on $\operatorname{dom}(\Phi) = \HH_0 \oplus \operatorname{dom}(\Phi_R')$, the following identities hold:
\[
A = \Phi^t M_A \Phi \quad \text{and} \quad \Phi J \Phi^t = J_{\mathcal{K}_{\mathrm{tot}}},
\]
where $M_A = 0_{\mathfrak{h}_0 \oplus \mathfrak{h}_0} \oplus M$, $M = \int_{[0,\infty)}^\oplus \lambda \cdot I_{\mathfrak{h}_\lambda \oplus \mathfrak{h}_\lambda} \, d\mu(\lambda)$, and $J_{\mathcal{K}_{\mathrm{tot}}} = J_{\mathfrak{h}_0} \oplus J_{\mathcal{K}}$ with $J_{\mathfrak{h}_0} = \begin{bmatrix} 0 & I_{\mathfrak{h}_0} \\ -I_{\mathfrak{h}_0} & 0 \end{bmatrix}$. 
Furthermore, $\Phi$ is injective with dense range, and the measure class $[\mu]$ and multiplicity function $\lambda \mapsto \dim \mathfrak{h}_\lambda$ are uniquely determined by $A$.
\end{theorem}

\begin{proof}
We prove the theorem by elaborating the ideas given in Section \ref{secproofoutline}. We construct the transformation $\Phi = \Phi_0 \oplus \Phi_R'$ in two distinct steps:

\smallskip
\noindent\textbf{Step 1: The Kernel Component $\Phi_0$.} \\
Since $J_0^2 = -I_{\HH_0}$ and $J_0^t = -J_0$, Theorem \ref{appen:thm_polar} yields a maximal isotropic closed subspace $\mathfrak{h}_0 \subset \HH_0$ such that 
\[
\HH_0 = \mathfrak{h}_0 \oplus J_0 \mathfrak{h}_0.
\]
Define the map $\Phi_0 : \HH_0 \to \mathfrak{h}_0 \oplus \mathfrak{h}_0$ by $\Phi_0(u + J_0 v) = (u, v)$ for $u, v \in \mathfrak{h}_0$. Then $\Phi_0$ is a bounded orthogonal isomorphism satisfying $\Phi_0^t = \Phi_0^{-1}$. Under $\Phi_0$, we have $\Phi_0 J_0 \Phi_0^t = J_{\mathfrak{h}_0}$, where $J_{\mathfrak{h}_0} = \begin{bmatrix} 0 & I_{\mathfrak{h}_0} \\ -I_{\mathfrak{h}_0} & 0 \end{bmatrix}$. Since $A|_{\HH_0} = 0$, we trivially have
\[
A|_{\HH_0} = \Phi_0^t \cdot 0 \cdot \Phi_0 = 0.
\]

\smallskip
\noindent\textbf{Step 2: The Range Component $\Phi_R'$.} \\
Note that the operator $A_R$ is bounded, positive, and injective operator on $\HH_R$. We work on $\HH_R$ henceforth and adjoin the kernel part at the end.

\begin{itemize}
    \item \underline{\textbf{Step 1:} The $A_R$-inner product space $\HA$ and the Riesz factorisation.} 
    \\ 
    Due to the positivity and injectivity of $A_R$, the bilinear form $\ipA{x}{y} := \ip{A_R x}{y}_{\HH_R}$ is an inner product on $\HH_R$, which need not be complete. Let $\HA$ denote its completion and $\iota: \HH_R\hookrightarrow\HA$ be the inclusion. Since for all $x \in \HH_R, \|\iota(x)\|_A^2 = \ip{A_R x}{x} \leq \|A_R\|\|x\|^2$, the map $\iota$ is bounded. Moreover, for $u\in\HA, y\in\HH_R, |\ipA{u}{\iota(y)}|\leq\|u\|_A\|y\|_A\leq
\|u\|_A\|A_R\|^{1/2}\|y\|$. Hence, the adjoint $\iota^t:\HA\to\HH_R$, defined by $\ip{\iota^t u}{y}_{\HH_R} = \ipA{u}{\iota y}$, for $u\in\HA$,
$y\in\HH_R$, satisfies $\|\iota^t\| \leq \|A_R\|^{1/2}$. Also since $\ip{\iota^t\iota x}{y} = \ipA{x}{y} = \ip{A_R x}{y}$ for all $x,y\in\HH_R$,
\begin{equation}
\label{eq:factor}
  A_R = \iota^t\iota.
\end{equation}
 In addition $\iota$ has a dense range in $\HA$: if $\ipA{u}{x} = 0$ for all $x\in\HH_R$, taking $u\in\HH_R$ gives $\|u\|_A^2 = 0$, so $u = 0$.

\item \underline{\textbf{Step 2: }Extending $T_R = J_R A_R$ to a bounded skew-symmetric injective operator}  \\ \underline{$\widehat{T}$ on $\HA$.} \\
$T_R$ is injective due to the injectivity of $J_R$ and $A_R$ on $\HH_R$. Also since $J_R$ is skew-symmetric on $\HH_R$, we have
\[
  \ipA{T_R x}{y}
  = \ip{A_R J_R A_R x}{y}
  = \ip{J_R A_R x}{A_R y}
  = -\ip{A_R x}{J_R A_R y}
  = -\ipA{x}{T_R y},
\] that is, the operator $T_R = J_RA_R$ is skew-symmetric with respect to $\ipA{\cdot}{\cdot}$. Moreover, since $J_R$ is orthogonal and
$\|A_R x\|^2 \leq \|A_R\|\ip{A_R x}{x} = \|A_R\|\|x\|_A^2$, we have
\[
  \|T_R x\|_A^2
  = \ip{A_R J_R A_R x}{J_R A_R x}
  \leq \|A_R\|\|A_R x\|^2
  \leq \|A_R\|^2\|x\|_A^2,
\] that is $T_R$ is a bounded operator on $\HH_R$. Hence, by the density of $\HH_R$ in $\HA$ and skew-symmetry, $T_R$ extends to $\widehat{T}:\HA\to\HA$ which is bounded, skew-symmetric and injective on $\HA$.

\item \underline{\textbf{Step 3:} Spectral decomposition of $\widehat{S} := -\widehat{T}^2$.}
\\
Since $\widehat{T}$ is bounded, skew-symmetric and injective, the operator $\widehat{S} = (\widehat{T})^t\widehat{T} = -\widehat{T}^2$ is bounded, positive and injective on $\HA$. Thus $\sigma(\widehat{S}) \subseteq [0,\infty)$. Since $\widehat{T}$ is skew-adjoint, it commutes with
$\widehat{S}$:
\begin{equation}
\label{eq:commute}
  \widehat{T}\,\widehat{S}
  = \widehat{T}(-\widehat{T}^2)
  = -\widehat{T}^3
  = (-\widehat{T}^2)\widehat{T}
  = \widehat{S}\,\widehat{T}.
\end{equation}
Now apply the spectral theorem \cite[Theorem~VIII.4]{RS81} and
multiplicity theory \cite[Section~VIII.2]{RS81} to $\widehat{S}$.
There exist a $\sigma$-finite Borel measure $\nu_0$ on $[0,\infty)$
and a measurable bundle $s\mapsto\widetilde{\fh}_s$ with
\begin{equation}
\label{eq:S-decomp}
  \HA \;\cong\;
  \int^\oplus_{[0,\infty)}\widetilde{\fh}_s\,d\nu_0(s),
  \qquad
  \widehat{S} \;\longleftrightarrow\;
  \int^\oplus_{[0,\infty)} s\cdot I_{\widetilde{\fh}_s}\,d\nu_0(s).
\end{equation}
By \eqref{eq:commute} and the commutant theorem for direct integrals
\cite[Theorem~XIII.83]{RS78}, $\widehat{T}$ is decomposable:
\begin{equation}
\label{eq:T-decomp}
  \widehat{T} \;\longleftrightarrow\;
  \int^\oplus_{[0,\infty)} \tau_s\,d\nu_0(s),
\end{equation}
where $\tau_s^t = -\tau_s$ and $-\tau_s^2 = s\cdot
I_{\widetilde{\fh}_s}$ on each fibre.

\item \underline{\textbf{Step 4:} Fibre-level structure and the isomorphism $\Gamma_0$.} \\
Since $s=0$ is an element of the continuous spectrum, the set \(\{0\}\) is null with respect to measure \(\nu_0\). Thus, we are justified in removing it from the domain of our direct integral decomposition henceforth. 
\\ Fix $s>0$. The operator $j_s := s^{-1/2}\tau_s$ satisfies $j_s^t = -j_s$ and $j_s^2= - I_{\widetilde{\fh}_s}$, making it a
\emph{complex structure} on $\widetilde{\fh}_s$. Therefore, we choose any closed subspace $\fh_s\subset\widetilde{\fh}_s$ maximal subject
to $\fh_s\perp j_s\fh_s$ when $s>0$ (see Theorem \ref{appen:thm_polar}). By maximality,
$\widetilde{\fh}_s = \fh_s\oplus j_s\fh_s$, and
\begin{equation}
\label{eq:gamma}
  \gamma_s:\fh_s\oplus\fh_s\xrightarrow{\;\sim\;}\widetilde{\fh}_s,
  \qquad
  \gamma_s(u,v) := u + j_s v,
\end{equation}
is a real orthogonal isomorphism (since $\langle u,j_s v\rangle = 0$ by maximality and $j_s$ is orthogonal). Under $\gamma_s$:
\begin{equation}
\label{eq:fiber-tau}
  \gamma_s^{-1}\,\tau_s\,\gamma_s = s^{1/2}\,J_{\fh_s},
  \qquad
  J_{\fh_s} :=
  \begin{pmatrix}0&I_{\fh_s}\\-I_{\fh_s}&0\end{pmatrix},
\end{equation}
since $\tau_s(u+j_s v) = s^{1/2}j_s(u+j_s v) = s^{1/2}(j_s u-v)$,
giving $\gamma_s^{-1}\tau_s\gamma_s(u,v) = s^{1/2}(-v,u) =
s^{1/2}J_{\fh_s}(u,v)$.

\noindent Now we move to the measure $\mu$. Fix $s\in (0,\infty)$, let $\lambda = s^{1/2}$, push $\nu_0$ forward under $s\mapsto s^{1/2}$ to a measure $\mu$ on $(0,\infty)$, set
$\fh_\lambda := \fh_{\lambda^2}$ and $\gamma_\lambda :=
\gamma_{\lambda^2}$. Assemble the fibrewise isomorphisms into a global real orthogonal isomorphism
\begin{equation}
\label{eq:Gamma0}
  \Gamma_0 : \HA \xrightarrow{\;\sim\;}
  \KK\oplus\KK :=\int^\oplus_{(0,\infty)}
  (\fh_\lambda\oplus\fh_\lambda)\,d\mu(\lambda).
\end{equation}
Under $\Gamma_0$, equations \eqref{eq:S-decomp}--\eqref{eq:T-decomp}
become:
\begin{align}
  \Gamma_0\,\widehat{S}\,\Gamma_0^{-1}
  &= M^2 := \int^\oplus_{(0,\infty)} \lambda^2\cdot
    I_{\fh_\lambda\oplus\fh_\lambda}\,d\mu(\lambda),
  \label{eq:S-fiberwise}\\
  \Gamma_0\,\widehat{T}\,\Gamma_0^{-1}
  &= M_J :=\int^\oplus_{(0,\infty)} \lambda\cdot J_{\fh_\lambda}\,d\mu(\lambda)
    = M\,\Jbf_\KK = \Jbf_\KK\,M,
  \label{eq:T-fiberwise}
\end{align}
where $M =\int^\oplus_{(0,\infty)} \lambda\cdot I_{\fh_\lambda\oplus\fh_\lambda}
\,d\mu(\lambda)$, and the commutativity $M\Jbf_\KK = \Jbf_\KK M$
holds fibrewise since both are diagonal in the direct integral
structure.

\item \underline{\textbf{Step 5:} Construction of $\Phi_R$.}
\\
Define the isometry
\begin{equation}
\label{eq:PhiR-def}
  \Phi_R := \Gamma_0\circ\iota : \HH_R \longrightarrow \KK\oplus\KK
\end{equation}
Since $\iota:\HH_R\to\HA$ and $\Gamma_0:\HA\to\KK\oplus\KK$ are both bounded, $\Phi_R$ is bounded. Since $\iota$ has dense range in
$\HA$ and $\Gamma_0$ is an isomorphism, $\Phi_R$ has dense range. Since $\iota$ is injective (as $\|\iota x\|_A = \|x\|_A > 0$ for $x\neq 0$) and $\Gamma_0$ is injective, $\Phi_R$ is injective. Next we show that $\Phi_R^t$ is densely defined. For $z = \Phi_R y
\in \ran(\Phi_R)$ and any $x\in\HH_R$:
\[
  |{\ipK{z}{\Phi_R x}}|
  = |\ipK{\Phi_R y}{\Phi_R x}|
  = |\ipA{y}{x}|
  \leq \|y\|_A\|x\|_A
  \leq \|y\|_A\|A_R\|^{1/2}\|x\|,
\]
so $x\mapsto\ipK{z}{\Phi_R x}$ is bounded for each
$z\in\ran(\Phi_R)$. Hence $\ran(\Phi_R)\subseteq\dom(\Phi_R^t)$.
Since $\ran(\Phi_R)$ is dense in $\KK$, $\dom(\Phi_R^t)$
is dense.
 Next we prove that $A_R = \Phi_R^t\Phi_R$. For all $x,y\in\HH_R$,
\begin{align*}
  \ip{\Phi_R^t\Phi_R x}{y}_{\HH_R}
  &= \ipK{\Phi_R x}{\Phi_R y}
   = \ipK{\Gamma_0\iota x}{\Gamma_0\iota y} \\
   &= \ipA{\iota x}{\iota y}
   = \ipA{x}{y} \\
   &= \ip{A_R x}{y}_{\HH_R}
\end{align*}
$\Rightarrow \Phi_R^t\Phi_R = A_R$.
\\
Now for $y = \Phi_R x\in\ran(\Phi_R)$, note that $\Phi_R^t y =
\Phi_R^t\Phi_R x = A_R x$. Then, using \eqref{eq:T-fiberwise}
in the form $\Gamma_0\widehat{T} = M_J\Gamma_0$:
\begin{align*}
  \Phi_R J_R\Phi_R^t y
  &= \Phi_R J_R A_R x
   = \Gamma_0\iota(J_R A_R x)
   = \Gamma_0\widehat{T}(\iota x) \\
   &= M_J\Gamma_0\iota x
   = M_J\Phi_R x \\
   &= M\Jbf_\KK y.
\end{align*}
Since $\ran(\Phi_R)$ is dense in $\KK$ and both sides
are continuous (the right side is bounded since $M$ and $\Jbf_\KK$ are bounded), the above identity extends to all of $\KK$, that is, $\Phi_R J_R\Phi_R^t = M\Jbf_\KK$.

\item \underline{\textbf{Step 6:} The equivalent Williamson's normal form}
\\
Define $\Phi_R^\prime := M^{-1/2}\Phi_R$. Then the domain
\[
  \dom(\Phi_R^\prime) := \{x\in\HH_R : \Phi_R x\in\ran(M^{1/2})\}.
\]
is dense. This can be seen as follows. Define the subspace $
\mathcal{D} := J_R \Phi_R^t \bigl(\operatorname{ran}(M^{1/2})\bigr).$ We first show that $\mathcal{D} \subseteq \operatorname{dom}(\Phi_R^\prime)$. Let $x \in \mathcal{D}$. Then $x = J_R \Phi_R^t M^{1/2} z$ for some $z \in \mathcal{K} \oplus \mathcal{K}$. Using the symplectic identity $\Phi_R J_R \Phi_R^t = M J_{\mathcal{K}}$ and the commutativity of $M$ and $J_{\mathcal{K}}$, we obtain:
$$
\Phi_R x = \Phi_R J_R \Phi_R^t M^{1/2} z = M J_{\mathcal{K}} M^{1/2} z = M^{1/2} (M J_{\mathcal{K}} z).
$$
Since $M J_{\mathcal{K}} z \in \mathcal{K} \oplus \mathcal{K}$, it follows that $\Phi_R x \in \operatorname{ran}(M^{1/2})$, and hence $x \in \operatorname{dom}(\Phi_R^\prime)$. Thus,
\begin{equation} \label{eq:D_subset_dom}
    \mathcal{D} \subseteq \operatorname{dom}(\Phi_R^\prime).
\end{equation}
Now it remains to show that $\mathcal{D}$ is dense in $H_R$. Since $M$ is positive and injective, $M^{1/2}$ is also injective, so
$
\overline{\operatorname{ran}(M^{1/2})} = (\ker M^{1/2})^\perp = \mathcal{K} \oplus \mathcal{K},
$ which implies that $\operatorname{ran}(M^{1/2})$ is dense in $\mathcal{K} \oplus \mathcal{K}$. Because $\Phi_R^t$ is a bounded linear operator, the image of a dense set under $\Phi_R^t$ is dense in $\operatorname{ran}(\Phi_R^t)$, so
\[
\overline{\Phi_R^t \bigl(\operatorname{ran}(M^{1/2})\bigr)} = \overline{\operatorname{ran}(\Phi_R^t)}.
\]
Furthermore, for any bounded operator between Hilbert spaces, $\overline{\operatorname{ran}(\Phi_R^t)} = (\ker \Phi_R)^\perp$. By the injectivity of $\Phi_R$, we have $\ker \Phi_R = \{0\}$, which yields
\[
\overline{\operatorname{ran}(\Phi_R^t)} = H_R.
\]
Finally, since $J_R$ is orthogonal, it preserves closures of dense subspaces, that is,
\[
\overline{\mathcal{D}} = \overline{J_R \Phi_R^t \bigl(\operatorname{ran}(M^{1/2})\bigr)} = J_R \overline{\Phi_R^t \bigl(\operatorname{ran}(M^{1/2})\bigr)} = J_R H_R = H_R.
\]
Combining this with \eqref{eq:D_subset_dom}, we conclude that $\overline{\operatorname{dom}(\Phi_R^\prime)} = H_R$, establishing that $\Phi_R^\prime$ is densely defined.
\\
Next we prove that the operator $\Phi_R^\prime$ is closable. Since $(M^{-1/2}\Phi_R)^t\supseteq\Phi_R^t M^{-1/2}$, we have $\dom((\Phi_R^\prime)^t)
\supseteq \dom(M^{-1/2}) = \ran(M^{1/2})$, which is dense. Hence $(\Phi_R^\prime)^t$ is densely defined and $\Phi_R^\prime$ is closable. 
\\ Finally, we observe the following. For $x,y\in\dom(\Phi_R^\prime)$, using
self-adjointness of $M^{-1/2}$ and $M^{-1/2}MM^{-1/2} = I$ on $\ran(M^{1/2})$:
\begin{align*}
  \ip{(\Phi_R^\prime)^t M\Phi_R^\prime x}{y}
  &= \ipK{M\Phi_R^\prime x}{\Phi_R^\prime y}
   = \ipK{M\cdot M^{-1/2}\Phi_R x}{M^{-1/2}\Phi_R y}\\
  &= \ipK{\Phi_R x}{M^{-1/2}MM^{-1/2}\Phi_R y} \\
  &= \ipK{\Phi_R x}{\Phi_R y} = \ipK{\Phi_R^t \Phi_Rx}{y} \\
  &= \ipK{A_Rx}{y}
\end{align*}
$\Rightarrow (\Phi_R^\prime)^t M\Phi_R^\prime = A_R$. For symplecticity, take $x\in\dom(M^{-1/2}) = \ran(M^{1/2})$. Then,  
\begin{align}
   M^{-1/2}\Phi_R J_R\Phi_R^t M^{-1/2}x
   = M^{-1/2}(M\Jbf_\KK)M^{-1/2}x.\notag
\end{align}
Since $M^{-1/2}M = M^{1/2}$ and $\Jbf_\KK$ commutes with $M$:
\begin{align}
  M^{-1/2}(M\Jbf_\KK)M^{-1/2}
  &= M^{-1/2}M^{1/2}\Jbf_\KK M^{1/2}M^{-1/2}
   = \Jbf_\KK
   \qquad\text{on }\ran(M^{1/2}).
  \label{eq:form1-symp-clean}
\end{align}
Hence $\Phi_R^\prime J_R(\Phi_R^\prime)^t = \Jbf_\KK$ on the dense domain
$\ran(M^{1/2})\subseteq\dom(\Phi_R^t)$.

\item \underline{\textbf{Step 8:} Uniqueness.}
\\
Suppose that
\[
\widetilde{\Phi}_R:H_R\longrightarrow\widetilde{\mathcal K}
\]
is another bounded operator with dense range satisfying
\[
\widetilde{\Phi}_R^t\widetilde{\Phi}_R=A_R,
\qquad
\widetilde{\Phi}_R J_R\widetilde{\Phi}_R^t
=
\widetilde M J_{\widetilde{\mathcal K}},
\]
where
\[
\widetilde M
=
\int_{(0,\infty)}^\oplus
\lambda I_{\widetilde{\mathfrak h}_\lambda}
\,d\widetilde\mu(\lambda).
\]

Define
\[
U_0:\operatorname{ran}(\Phi_R)
\longrightarrow
\operatorname{ran}(\widetilde{\Phi}_R),
\qquad
U_0(\Phi_Rx):=\widetilde{\Phi}_Rx.
\]
Since $\Phi_R$ is injective, $U_0$ is well-defined. Moreover,
\[
\begin{aligned}
\|U_0\Phi_Rx\|_{\widetilde{\mathcal K}}^2
&=
\|\widetilde{\Phi}_Rx\|_{\widetilde{\mathcal K}}^2\\
&=
\langle
\widetilde{\Phi}_R^t\widetilde{\Phi}_Rx,x
\rangle_{H_R}\\
&=
\langle A_Rx,x\rangle_{H_R}\\
&=
\langle
\Phi_R^t\Phi_Rx,x
\rangle_{H_R}\\
&=
\|\Phi_Rx\|_{\mathcal K}^2.
\end{aligned}
\]
Hence $U_0$ is an isometry. Since both $\operatorname{ran}(\Phi_R)$ and
$\operatorname{ran}(\widetilde{\Phi}_R)$ are dense, $U_0$ extends uniquely to a unitary operator
\(
U:\mathcal K\longrightarrow\widetilde{\mathcal K}.
\)
By construction,
\(
\widetilde{\Phi}_R=U\Phi_R.
\)
Consequently,
$$
\widetilde M J_{\widetilde{\mathcal K}} =
\widetilde{\Phi}_R J_R\widetilde{\Phi}_R^t =
U\Phi_RJ_R\Phi_R^tU^t = U(MJ_{\mathcal K})U^t.
$$
Since $M$ commutes with $J_{\mathcal K}$,
$J_{\mathcal K}^t=-J_{\mathcal K}$, and
$J_{\mathcal K}^2=-I$, we have
$$
(MJ_{\mathcal K})^t(MJ_{\mathcal K}) =
(-J_{\mathcal K}M)(MJ_{\mathcal K}) =
-J_{\mathcal K}M^2J_{\mathcal K} = M^2.
$$ Therefore \(
|MJ_{\mathcal K}|=M. \) Similarly,
\(
|\widetilde M J_{\widetilde{\mathcal K}}|
=
\widetilde M.
\) Using the unitary invariance of the functional calculus we have,
\[
\widetilde M
=
|\widetilde M J_{\widetilde{\mathcal K}}|
=
U|MJ_{\mathcal K}|U^t
=
UMU^t.
\]
Thus $M$ and $\widetilde M$ are unitarily equivalent. By the uniqueness of spectral measure class and multiplicity function for self-adjoint operators
\cite[Section~VIII.2]{RS81}, $[\mu] = [\widetilde\mu]$ and $\dim\fh_\lambda = \dim\widetilde\fh_\lambda$ $\mu$-a.e.
\end{itemize}

\smallskip
\noindent\textbf{Step 3: Direct Sum Construction on $\mathcal{H} = \HH_0 \oplus \mathcal{H}_R$.} \\
Define $\mathcal{K}_{\mathrm{tot}} = (\mathfrak{h}_0 \oplus \mathfrak{h}_0) \oplus (\mathcal{K} \oplus \mathcal{K})$ and the operator 
\[
\Phi = \Phi_0 \oplus \Phi_R' : \HH_0 \oplus \operatorname{dom}(\Phi_R') \subset \mathcal{H} \longrightarrow \mathcal{K}_{\mathrm{tot}}.
\]
Since $\Phi_0$ is bounded and $\operatorname{dom}(\Phi_R')$ is dense in $\mathcal{H}_R$, $\operatorname{dom}(\Phi) = \HH_0 \oplus \operatorname{dom}(\Phi_R')$ is dense in $\mathcal{H}$. Taking the transpose (adjoint) in block-diagonal form, $\Phi^t = \Phi_0^t \oplus (\Phi_R')^t$.

Combining the kernel and range results:
\[
\Phi^t M_A \Phi = \begin{bmatrix} \Phi_0^t & 0 \\ 0 & (\Phi_R')^t \end{bmatrix} \begin{bmatrix} 0 & 0 \\ 0 & M \end{bmatrix} \begin{bmatrix} \Phi_0 & 0 \\ 0 & \Phi_R' \end{bmatrix} = \begin{bmatrix} 0 & 0 \\ 0 & (\Phi_R')^t M \Phi_R' \end{bmatrix} = 0 \oplus A_R = A.
\]
Similarly, for the symplectic structure:
\[
\Phi J \Phi^t = \begin{bmatrix} \Phi_0 J_0 \Phi_0^t & 0 \\ 0 & \Phi_R' J_R (\Phi_R')^t \end{bmatrix} = \begin{bmatrix} J_{\mathfrak{h}_0} & 0 \\ 0 & J_{\mathcal{K}} \end{bmatrix} = J_{\mathcal{K}_{\mathrm{tot}}}.
\]
Injectivity and density of $\mathrm{ran}(\Phi)$ follow directly from the injectivity and dense ranges of $\Phi_0$ and $\Phi_R^\prime$. 
\end{proof}

\begin{remark}[The two forms and the $A$-inner product]
\label{rem:two-forms}
The proof makes explicit why the two forms of the normal form are
not equivalent choices of presentation but reflect two different
isometric structures.
\end{remark}

\begin{remark}[Finite-dimensional result]
\label{rem:finite-dim}
When $\mathcal{H}$ is finite-dimensional, $\mu = \sum_i
m_i\delta_{\lambda_i}$ with $\lambda_i$ the symplectic eigenvalues and $m_i$ their multiplicities and $\delta_{\lambda_i}$ is the Dirac delta measure concentrated at $\lambda_i$. The operators decompose to
$A\leftrightarrow \bigoplus_i(\lambda_i I_{2m_i})$ under a symplectic operator \(S\), recovering the
classical Williamson's normal form.
\end{remark}

\section{Simultaneous Williamson's Normal Form}
\label{sec:sim-wnf}

\subsection{Discussion on the main result and outline of the proof}
\label{sec:sim-proofoutline}

We now extend the direct-integral Williamson's normal form established in Section~\ref{sec:direct-integral} to a pair of commuting positive operators. Let $A$ and $B$ be bounded positive operators on $\HH$ satisfying hypotheses (H1) and (H2) below. The joint kernel $\HH_0 = \ker(A)\cap\ker(B)$ is $J$-invariant by (H1), so $J=J_0\oplus J_R$ where $J_R=J|_{\HH_R}$ and $\HH_R=\HH_0^\perp$. The central object is the Hilbert space $\HC$ obtained by completing $\HH_R$ in the inner product induced by $C_R:=A_R+B_R$. The operators $T_A:=J_RA_R$, $T_B:=J_RB_R$, and $T_C:=J_RC_R$ extend to bounded skew-symmetric operators on $\HC$. By (H2), they commute and are simultaneously diagonalizable. The spectral theorem applied to $\widehat{S}_C:=-\widehat{T}_C^2$ yields a direct integral decomposition in which the fibre-level analysis reduces to diagonalizing commuting positive operators $a_\lambda$ and $b_\lambda$ satisfying $a_\lambda+b_\lambda=\lambda I$. A second-stage diagonalization of $a_\lambda$ within each fibre yields the final common spectral representation. The symplectic transformation $\Phi$ is then constructed as an isometry from $(\HH_R,\ipC{\cdot}{\cdot})$ to a direct integral space, followed by a scaling that incorporates the symplectic spectrum. The final result, Theorem~\ref{thm:sim-wnf}, provides a simultaneous symplectic representation of the form
\[
A=\Phi^t M_A\Phi,
\qquad
B=\Phi^t M_B\Phi,
\qquad
\Phi J\Phi^t=J_{\mathcal K_{\mathrm{tot}}},
\]
where $M_A$ and $M_B$ are multiplication operators on a common spectral representation space.

\subsection{Preliminaries}

Let $A$ and $B$ be two bounded positive operators on $\HH$. We define the joint kernel and its orthogonal complement as follows:
\begin{align}
  \HH_0 &:= \ker(A) \cap \ker(B), \label{eq:H0}\\
  \HH_R &:= \HH_0^\perp = \overline{\ran(A) + \ran(B)}. \label{eq:HR}
\end{align}
Restrictions to $\HH_R$ are denoted by
\[
A_R := A|_{\HH_R},\quad B_R := B|_{\HH_R},\quad 
J_R := J|_{\HH_R},\quad C_R := A_R + B_R.
\]
The core of the main result lies in the following two hypotheses.
\begin{enumerate}
  \item[(H1)] \textbf{$J$-invariance of the joint kernel.} 
    $J(\HH_0) \subseteq \HH_0$. 
    
  \item[(H2)] \textbf{Commutativity condition.} 
    The operators
    \[
    T_A := J_R A_R,\qquad T_B := J_R B_R
    \]
    commute, that is, $T_A T_B = T_B T_A$.  
\end{enumerate}
As a consequence, $J=J_0\oplus J_R$ where $J_0:=J|_{\HH_0}$ and $J_R:=J|_{\HH_R}$, and $C_R$ is bounded, positive, and injective on $\HH_R$. Moreover, $T_C:=J_RC_R=T_A+T_B$ and $T_A,T_B$ each commute with $T_C$.

Proceeding as in the proof of Theorem~\ref{thm:di-wnf}, the bilinear form
\[
\ipC{x}{y} := \ip{C_R x}{y}_{\HH_R}, \qquad x,y \in \HH_R,
\]
is an inner product on $\HH_R$. Let $\HC$ be its completion, and let $\iota : \HH_R \hookrightarrow \HC$ be the inclusion. The adjoint $\iota^t : \HC \to \HH_R$ satisfies $C_R = \iota^t \iota$.

\subsection{Main theorem}
\label{sec:theorem}

\begin{theorem}[Simultaneous Williamson's normal form]
\label{thm:sim-wnf}
Let $A, B \in \mathcal{B}(\HH)$ be bounded positive operators on the 
real symplectic Hilbert space $(\HH, J)$ satisfying (H1) and (H2).  
Then there exist a $\sigma$-finite Borel measure $\mu_{\mathrm{tot}}$ on an index space $\Lambda_{\mathrm{tot}} = \{*\} \sqcup \Lambda_+$, where $\Lambda_+ \subseteq (0,\infty)\times[0,1]$, a measurable Hilbert space bundle $\lambda \mapsto \fh_\lambda$ over $\Lambda_{\mathrm{tot}}$, a direct integral Hilbert space
\[
\KKtot:=
(\KK_0\oplus\KK_0)\oplus
\int^\oplus_{\Lambda_+} (\fh_{(\lambda,\alpha)}\oplus\fh_{(\lambda,\alpha)}) \, d\mu_+(\lambda,\alpha),
\]
with $\KK_0:=\fh_*$, fibrewise symplectic operator
\[
\Jbf_{\KKtot} := \int^\oplus_{\Lambda_{\mathrm{tot}}} 
\begin{pmatrix} 0 & -I_{\fh_\lambda} \\ I_{\fh_\lambda} & 0 \end{pmatrix}
d\mu_{\mathrm{tot}}(\lambda),
\]
and multiplication operators
\[
M_A := \int^\oplus_{\Lambda_{\mathrm{tot}}} \alpha(\lambda) I_{\fh_\lambda\oplus\fh_\lambda} \, d\mu_{\mathrm{tot}}(\lambda),\qquad
M_B := \int^\oplus_{\Lambda_{\mathrm{tot}}} \beta(\lambda) I_{\fh_\lambda\oplus\fh_\lambda} \, d\mu_{\mathrm{tot}}(\lambda),
\]
where $\alpha(*)=\beta(*)=0$ and for $(\lambda,t)\in\Lambda_+$, $\alpha=t\lambda$, $\beta=(1-t)\lambda$, such that there exists a closed, densely defined operator
\[
\Phi : \dom(\Phi) \subset \HH \longrightarrow \KKtot
\]
satisfying on $\dom(\Phi)$:
\begin{align}
  A &= \Phi^t M_A \Phi, \label{eq:nf-A}\\
  B &= \Phi^t M_B \Phi, \label{eq:nf-B}\\
  \Phi J \Phi^t &= \Jbf_{\KKtot}. \label{eq:symp1}
\end{align}
Moreover:
\begin{enumerate}
  \item[(i)] $\Phi$ is injective and has dense range.
  \item[(ii)] $\dom(\Phi) = \HH_0 \oplus \dom(\Phi_R^\prime)$ where 
    $\Phi|_{\HH_0}$ is bounded and $\dom(\Phi_R^\prime)$ is dense in $\HH_R$.
  \item[(iii)] The joint measure class $[\mu_{\mathrm{tot}}]$ and the 
    multiplicity function $\lambda \mapsto \dim\fh_\lambda$ are uniquely 
    determined by $(A,B)$.
\end{enumerate}
\end{theorem}

\begin{proof}
The proof has five stages.

\underline{\textbf{Stage I:} Extending $T_A, T_B, T_C$ to $\HC$.}
\\
For $x \in \HH_R$,
\[
\|T_C x\|_C^2 = \ipC{J_R C_R x}{J_R C_R x}
= \ip{C_R J_R C_R x}{J_R C_R x}_{\HH_R} \leq \|C_R\|^2 \|x\|_C^2,
\]
so $T_C$ is bounded. Also, for $x,y \in \HH_R$,
\[
\ipC{T_C x}{y} = \ip{C_R J_R C_R x}{y}_{\HH_R}
= -\ipC{x}{T_C y},
\]
so $T_C$ is skew-symmetric. Hence, by the density of $\HH_R$ in $\HC$, $T_C$ extends to a bounded, skew-symmetric, and injective operator $\widehat{T}_C$ on $\HC$. Similarly, using $A_R \leq C_R$ and $B_R \leq C_R$, $T_A$ and $T_B$ extend to bounded operators $\widehat{T}_A$ and $\widehat{T}_B$ on $\HC$. By linearity, density, and the commutativity hypothesis, we have
\[
\widehat{T}_A + \widehat{T}_B = \widehat{T}_C,\qquad
\widehat{T}_A \widehat{T}_B = \widehat{T}_B \widehat{T}_A,
\]
and $\widehat{T}_A,\widehat{T}_B$ commute with $\widehat{S}_C := -\widehat{T}_C^2$.

\underline{\textbf{Stage II:} Spectral decomposition of $\widehat{S}_C$.}
\\
Define $\widehat{S}_C := -\widehat{T}_C^2 = \widehat{T}_C^t\widehat{T}_C$. This is a bounded, positive, self-adjoint, injective operator on $\HC$ with spectrum in $[0,\infty)$. Applying the spectral theorem \cite[Theorem~VIII.4]{RS81} and multiplicity theory \cite[Section~VIII.2]{RS81} to $\widehat{S}_C$, there exist a $\sigma$-finite Borel measure $\mu_1$ on $[0,\infty)$ and an orthogonal isomorphism
\[
\Gamma_0 : \HC \xrightarrow{\sim} 
\KK := \int^\oplus_{[0,\infty)} \widetilde{\fh}_\lambda \, d\mu_1(\lambda)
\]
such that $\Gamma_0 \widehat{S}_C \Gamma_0^{-1} = \Lambda_C^2$ (multiplication by $\lambda^2$). Since $\widehat{T}_A$ and $\widehat{T}_B$ commute with $\widehat{S}_C$, they are decomposable by the commutant theorem for direct integrals \cite[Theorem~XIII.83]{RS78}:
\[
\Gamma_0 \widehat{T}_A \Gamma_0^{-1} = \int^\oplus_{[0,\infty)} \tau^\lambda_A \, d\mu_1(\lambda),\qquad
\Gamma_0 \widehat{T}_B \Gamma_0^{-1} = \int^\oplus_{[0,\infty)} \tau^\lambda_B \, d\mu_1(\lambda).
\]
On each fibre, $\tau^\lambda_C=\tau^\lambda_A+\tau^\lambda_B$, $[\tau^\lambda_A,\tau^\lambda_B]=0$, and $(\tau^\lambda_C)^t\tau^\lambda_C=\lambda^2 I$.

\underline{\textbf{Stage III:} Fibre-level analysis and splitting.}
\\
As before, we omit the $\{0\}$ from our domain since it has measure zero. Fix $\lambda>0$ and set $j_\lambda := \lambda^{-1}\tau^\lambda_C$. From $(\tau^\lambda_C)^t\tau^\lambda_C = \lambda^2 I$ and $(\tau^\lambda_C)^t = -\tau^\lambda_C$, we have $j_\lambda^t=-j_\lambda$ and $j_\lambda^2=-I_{\widetilde{\fh}_\lambda}$, so $j_\lambda$ is a complex structure. By the standard complex structure polarisation (Theorem~\ref{appen:thm_polar}), there exists a splitting $\widetilde{\fh}_\lambda = \fh_\lambda \oplus j_\lambda \fh_\lambda$ and an isomorphism $\gamma_\lambda : \fh_\lambda \oplus \fh_\lambda \xrightarrow{\sim} \widetilde{\fh}_\lambda$ such that $\gamma_\lambda^{-1} j_\lambda \gamma_\lambda = J_{\fh_\lambda}$. Define $a_\lambda := -j_\lambda \tau^\lambda_A$ and $b_\lambda := -j_\lambda \tau^\lambda_B$. Then
\[
\gamma_\lambda^{-1} \tau_A^\lambda \gamma_\lambda = J_{\fh_\lambda}a_\lambda,\qquad
\gamma_\lambda^{-1} \tau_B^\lambda \gamma_\lambda = J_{\fh_\lambda}b_\lambda,
\]
and $a_\lambda,b_\lambda$ are bounded, positive, self-adjoint operators satisfying
\[
a_\lambda + b_\lambda = \lambda I,\qquad
[a_\lambda, b_\lambda] = 0,\qquad
[a_\lambda, j_\lambda] = [b_\lambda, j_\lambda] = 0.
\]

\underline{\textbf{Stage IV:} Second-stage diagonalisation (fibrewise).}
\\
We now diagonalise $a_\lambda$ within each fibre. In the $\gamma_\lambda$-basis, $a_\lambda$ is a positive operator commuting with $J_{\fh_\lambda}$. Applying the real spectral theorem, since the spectral projections commute with $J_{\fh_\lambda}$, $J_{\fh_\lambda}$ decomposes fibrewise as the canonical symplectic operator. Thus, for each $\lambda>0$, there exist a measure $\nu_\lambda$ on $[0,\lambda]$ and an isomorphism
\[
V_\lambda : \widetilde{\fh}_\lambda \xrightarrow{\sim}
\int^\oplus_{[0,\lambda]} (\fh^{(\lambda)}_\alpha \oplus \fh^{(\lambda)}_\alpha) \, d\nu_\lambda(\alpha)
\]
such that
\[
a_\lambda \leftrightarrow \int^\oplus_{[0,\lambda]} \alpha\cdot I\,d\nu_\lambda(\alpha),\qquad
b_\lambda \leftrightarrow \int^\oplus_{[0,\lambda]} (\lambda-\alpha)\cdot I\,d\nu_\lambda(\alpha),\qquad
j_\lambda \leftrightarrow \int^\oplus_{[0,\lambda]} J_{\fh^{(\lambda)}_\alpha}\,d\nu_\lambda(\alpha).
\]

\underline{\textbf{Stage V:} Defining the transformation $\Phi$.}
\\
The joint index set and measure are
\[
\Lambda_+ := \bigsqcup_{\lambda>0} \{(\lambda,\alpha) : \alpha \in \spect(a_\lambda)\},\qquad
d\mu_+(\lambda,\alpha) := d\nu_\lambda(\alpha)\,d\mu_1(\lambda).
\]
The total direct integral space is
\[
\KKtot:=
(\KK_0\oplus\KK_0)\oplus
\int^\oplus_{\Lambda_+} (\fh_{(\lambda,\alpha)}\oplus\fh_{(\lambda,\alpha)}) \, d\mu_+(\lambda,\alpha),
\]
where $\KK_0 := \fh_*$ encodes the kernel.

For the kernel piece, choose $\fh_* \subset \HH_0$ maximal with $\fh_* \perp J_0\fh_*$ and define
\[
\Phi_0 := \gamma_*^{-1} : \HH_0 \xrightarrow{\sim} \fh_*\oplus\fh_* = \KK_0\oplus\KK_0,
\qquad \gamma_*(u,v) := u + J_0 v.
\]
Then $\Phi_0 J_0 \Phi_0^{-1} = J_{\fh_*}$ and $A|_{\HH_0}=B|_{\HH_0}=0$.

For the range piece, define bounded operators:
\begin{align}
  \Phi_R^{(0)} &:= \Gamma_0 \circ \iota : \HH_R \to \KK, \\
  \Phi_R^{(1)} &:= \left(\int^\oplus_{(0,\infty)} V_\lambda \, d\mu_1(\lambda)\right) \circ \Phi_R^{(0)} : \HH_R \to \KKtot\oplus\KKtot.
\end{align}
Then $\Phi_R^{(1)}$ satisfies
\[
(\Phi_R^{(1)})^t \Phi_R^{(1)} = C_R,\qquad
\Phi_R^{(1)} J_R (\Phi_R^{(1)})^t = M_C \Jbf_{\KKtot},
\]
\[
(\Phi_R^{(1)})^t M_A M_C^{-1} \Phi_R^{(1)} = A_R,\qquad
(\Phi_R^{(1)})^t M_B M_C^{-1} \Phi_R^{(1)} = B_R,
\]
where $M_A=\int_{\Lambda_+}^{\oplus}\alpha\,d\mu_+$, $M_B=\int_{\Lambda_+}^{\oplus}(\lambda-\alpha)\,d\mu_+$, $M_C=M_A+M_B$, and $\Jbf_{\mathcal{K}_{\mathrm{tot}}}=\int_{\Lambda_+}^{\oplus}\Jbf_{\fh_{\alpha}^{(\lambda)}}\,d\mu_+$. To verify the third identity, take $U = \left(\int^{\oplus}_{(0,\infty)}V_\lambda\,d\mu_1(\lambda)\right)\circ\Gamma_0$. Then $U\widehat{T}_A U^{-1}=\Jbf_{\KKtot} M_A$ and $U\widehat{T}_C U^{-1}=\Jbf_{\KKtot} M_C$, so for $x,y\in\HH_R$,
\begin{align*}
\langle (\Phi_R^{(1)})^t M_A M_C^{-1}\Phi_R^{(1)}x,y\rangle_{\HH_R}
&=\langle \iota x,\widehat{T}_{A}\widehat{T}_{C}^{-1}\iota y\rangle_{C}\\
&=\langle x,A_R y\rangle_{\HH_R}.
\end{align*}

Now define the (generally unbounded) transformation
\[
\Phi_R^\prime := M_C^{-1/2} \Phi_R^{(1)},\qquad
\dom(\Phi_R^\prime) := \{ x \in \HH_R : \Phi_R^{(1)} x \in \ran(M_C^{1/2}) \}.
\]
The domain is dense in $\HH_R$ by an argument analogous to Step 6 in the proof of Theorem~\ref{thm:di-wnf}. Finally, define
\[
\Phi := \Phi_0 \oplus \Phi_R^\prime,\qquad
\dom(\Phi) := \HH_0 \oplus \dom(\Phi_R^\prime).
\]

For $A$, on $\HH_0$, $\Phi_0^t \cdot 0 \cdot \Phi_0 = 0 = A|_{\HH_0}$; and on $\dom(\Phi_R^\prime)$:
\[
(\Phi_R^\prime)^t M_A \Phi_R^\prime = (\Phi_R^{(1)})^t M_A M_C^{-1} \Phi_R^{(1)} = A_R.
\]
Similarly, $(\Phi_R^\prime)^t M_B \Phi_R^\prime = B_R$.

For $J$, on $\HH_0$, $\Phi_0 J_0 \Phi_0^t = \Jbf_{\KK_0}$; and on $\ran(\Phi_R^\prime)$:
\[
\Phi_R^\prime J_R (\Phi_R^\prime)^t = M_C^{-1/2} (M_C \Jbf_{\KKtot}) M_C^{-1/2} = \Jbf_{\KKtot}.
\]
Since $\ran(\Phi_R^\prime)$ is dense, the identity extends to all of $\KKtot$.

The uniqueness of the joint measure class $[\mu_{\mathrm{tot}}]$ and the multiplicity function follows from standard spectral theory for commuting self-adjoint operators \cite[Section~VIII.2]{RS81}. This completes the proof.
\end{proof}

\section{Discussion on the main theorem}
\label{sec:discussion}
In this section, we make a discussion on the main result and see how it coincides with the finite-dimensional result in \cite{KAMAT2026298}.

\begin{enumerate}
    \item \textbf{Necessity of the commutativity hypothesis (H2)}: If a simultaneous symplectic diagonalisation exists, then on each fibre 
$(\lambda, t)$ we have $T_A \leftrightarrow t\lambda J$ and $T_B \leftrightarrow (1-t)\lambda J$, 
so $T_AT_B = T_BT_A = -t(1-t)\lambda^2 I$. Hence, the commutativity hypothesis (H2) is necessary.

\item Using $C_R = A_R + B_R$ as the primary operator is essential for symmetry. In the $C$-inner product, both $A_R$ and $B_R$ become bounded relative to $C_R$ 
($A_R \le C_R$, $B_R \le C_R$), so their $J$-rotations $T_A, T_B$ extend to $\HC$. Moreover, both commute with $T_C$, allowing simultaneous decomposition in Stage~II. Stage~IV then diagonalizes $A$ (and automatically $B$) within each $C$-fibre.

\item \textbf{The joint symplectic spectrum}: The parameters $(\lambda, t)$ live in $[0,\infty) \times [0,1]$, where:
\begin{itemize}
  \item $\lambda$ is a symplectic eigenvalue of $C = A+B$,
  \item $t = \alpha/\lambda$ is the fraction of that eigenvalue coming from $A$.
\end{itemize}
The constraint $\alpha + \beta = \lambda$ is automatic. The measure $\mu_+$ encodes both the distribution of $\lambda$ and the distribution 
of $t$ within each $\lambda$-eigenspace.

\item Setting $B=0$: $C_R = A_R$ and the $(A+B)$-inner product becomes the $A$-inner product.  
Then $b_\lambda = 0$ on every fibre, so $\beta = 0$ and $M_B = 0$. The construction reduces to the Williamson's normal form for a positive operator on a symplectic Hilbert space.

\item \textbf{Finite-dimensional result:} When $\mathcal{H}$ is finite-dimensional, $\mu = \sum_i
m_i\delta_{\lambda_i}$ and $\eta=\sum_j n_j\delta_{\nu_j}$ with $\lambda_i,\nu_j$ being the symplectic eigenvalues of $A$ and $B$ with multiplicities $m_i, n_j$, respectively. The operators decompose to $A\leftrightarrow \bigoplus_i(\lambda_i I_{2m_i})$ and \(B\leftrightarrow \bigoplus_j(\nu_j I_{2m_j})\) under a symplectic operator \(S\), recovering the finite-dimensional simultaneous Williamson's normal form \cite[Theorem~3.1]{KAMAT2026298}.
\end{enumerate}

\section{Applications}
In classical mechanics, $A$ and $B$ could represent two quadratic Hamiltonians (e.g., kinetic and potential energy). The simultaneous diagonalisation yields normal coordinates in which both Hamiltonians take the simple form $\sum \mu_i (p_i^2+q_i^2)$ and $\sum \nu_i (p_i^2+q_i^2)$. The commutativity condition (H2) means the two Hamiltonians generate flows that commute, which is a strong integrability condition.

\section{Concluding Remarks and Future Problems}

We have established a simultaneous Williamson's normal form for two bounded positive operators $A, B$ on a real symplectic Hilbert space, under the assumption that they commute in the sense that 
\[
T_A T_B = T_B T_A,
\]
where $T_A = JA$ and $T_B = JB$. The main result gives a common spectral representation in which both operators act as multiplication operators, and the symplectic structure reduces fiberwise to the standard rotation. This extends the finite-dimensional version from \cite{KAMAT2026298} to the infinite-dimensional setting.

That said, the unbounded case is a natural next step, and it comes with a number of genuine difficulties. First, there are immediate domain problems. If $A$ and $B$ are unbounded, the inner product 
\[
\langle x, y \rangle_C = \langle C x, y \rangle
\]
only makes sense for $x, y \in \operatorname{dom}(C)$, where $C = A+B$. The completion of this domain with respect to the associated norm need not coincide with the original Hilbert space $\mathcal H$. Moreover, the operators $T_A$ and $T_B$ become unbounded as well, so one has to be careful about where they are defined and how they interact.

On the spectral side, while the spectral theorem for unbounded self-adjoint operators still gives a direct integral decomposition, the spectral measures involved may be unbounded, and the commutativity condition has to be formulated in terms of spectral projections or resolvents rather than the operators themselves. This changes the nature of the problem considerably.

There is also the question of closability. The symplectic transformation $\Phi$ that appears in the construction may initially only be defined on a dense subspace, and it is not clear whether it extends to a closed operator without imposing extra hypotheses. Similarly, the joint kernel 
\[
\mathcal H_0 = \ker(A) \cap \ker(B)
\]
may fail to be closed when the operators are unbounded, so the usual orthogonal decomposition 
\[
\mathcal H = \mathcal H_0 \oplus \mathcal H_R
\]
may need to be rethought. And even on the kernel piece, the induced complex structure $J_0$ can be unbounded, which complicates the construction of the corresponding factor $\Phi_0$.

There are a few plausible routes around these obstacles. One could work with resolvents to bypass domain issues, or use quadratic form methods to handle the inner products indirectly. Another option is to define the transformation on a dense core first and then try to extend it by closure. If $T_A$ and $T_B$ generate strongly continuous groups, semigroup techniques might help. And of course, one could adapt the direct integral theory for unbounded operators, as discussed in \cite{Dixmier:1981:VNA,Hall:2013:QTM}.

Beyond the unbounded case, there are several other directions worth exploring. For instance, the results could be applied to quadratic Hamiltonians in quantum mechanics, or one could look for joint spectral invariants for non-commuting pairs. It would also be interesting to generalize the normal form to finite collections of pairwise commuting positive operators, or to develop numerical methods for approximating the joint spectral measure. There may even be applications in continuous-variable quantum information that have not yet been examined.

In any case, the simultaneous Williamson's normal form seems to offer a useful framework for infinite-dimensional symplectic diagonalization, and we see the extension to unbounded operators as one of the more challenging and worthwhile problems left open.

\section*{Acknowledgements}
Anmary Tonny thanks Dr. Sumesh K., Department of Mathematics, IIT Madras, for the financial support through Project No.   SB22231267MAETWO008573 and Project No. RF25260629MARFER008769 and Prof. B. V. Rajarama Bhat, ISI Bangalore for the financial support through the J.C. Bose Fellowship (Project No. JBR/2021/000024). This work was completed during her tenure at IIT Madras and ISI Bangalore..

\section{Appendix}
\subsection{Polarisation by a Complex Structure}
\begin{theorem}\label{appen:thm_polar}
    Let \((\HH,\langle\cdot,\cdot\rangle)\) be a real Hilbert space, and let \(J\in\mathcal{B}(\HH)\) be a bounded operator such that 
    \begin{enumerate}
        \item \(J^2=-I\) (complex structure), and
        \item \(J\) is orthogonal i.e. \(\langle Ju,Jv\rangle=\langle u,v\rangle\), for all \(u,v\in\HH\) (equivalently, \(J^t=-J\)). \end{enumerate}
        Define the associated symplectic form 
        \[\omega(u,v)\coloneqq\langle u,Jv\rangle, \text{ for all } u,v\in\HH.\]
        Then there exists a closed linear subspace \(V\subset\HH\) such that 
        \begin{enumerate}
            \item \(V\) is isotropic i.e. \(\omega(u,v)=0\), for all \(u,v\in V\), 
            \item \(V\) is maximal with respect to property (1) under inclusion, and
            \item the Hilbert space polarises into the following decomposition: 
            \[\HH=V\oplus JV.\]
        \end{enumerate}
\end{theorem}
\begin{proof}
    Define the family \[\mathcal{F}=\{V\subset\HH\;\text{is a closed subspace}\;:\:\omega(u,v)=0, \text{ for all } u,v\in V\}.\]
    Take a chain \(\{V_{\alpha}\}\), under inclusion, of elements of \(\mathcal{F}\). Then \(\overline{\bigcup_{\alpha}V_\alpha}\) is the upper bound of this chain and by Zorn's lemma, we obtain a maximal isotropic subspace \(V\). 
    Now we prove the Hilbert space decomposition. Suppose, for the sake of contradiction, that there exists a non-zero \(x\in(V\oplus JV)^{\perp}\). Note that such an element is orthogonal to \(V\) with respect to \(\omega\). Define 
    \[W=V\oplus\text{span}\{x\}.\]
    Then for any \(u+sx,v+tx\in W\), 
    \begin{align*}
        \omega(u+sx,v+tx)&=\omega(u,v)+t\omega(u,x)+s\omega(x,v)+st\omega(x,x)\\&=0,
    \end{align*}
    where the first three terms vanish because \(V\) is isotropic and \(x\) is \(\omega\)-orthogonal to \(V\). But this contradicts the maximality of \(V\). 
    Therefore, \(V\oplus JV\) must be the whole space \(\HH\).
\end{proof}


\bibliographystyle{plain}  
\bibliography{references}

\end{document}